\documentclass[12pt]{amsart}
\usepackage{amsmath,amsfonts,amssymb,amscd,amsthm}
\usepackage {pstricks}
\usepackage{pstricks,pst-node}
\usepackage{tikz}

\numberwithin{equation}{section}

\newtheorem{lemma}[equation]{Lemma}
\newtheorem*{theorem*}{Theorem}
\newtheorem*{lemma*}{Lemma}
\newtheorem{theorem}[equation]{Theorem}
\newtheorem{proposition}[equation]{Proposition}
\newtheorem{definition}[equation]{Definition}
\newtheorem{corollary}[equation]{Corollary}
\theoremstyle{remark}
\newtheorem{remark}[equation]{Remark}

\DeclareMathOperator{\Id}{Id}
\DeclareMathOperator{\im}{im}

\DeclareMathOperator{\End}{End}

\DeclareMathOperator{\Rad}{Rad}

\newcommand{\C}{{\mathbb C}}

\newcommand{\Z}{{\mathbb Z}}

\newcommand{\CC}{{\mathcal C}}
\newcommand{\CB}{{\mathcal B}}

\newcommand{\CS}{{\mathcal S}}
\newcommand{\CT}{{\mathcal T}}

\newcommand{\fsl}{\mathfrak{sl}}

\newcommand{\Sym}{{\mathrm{Sym}}}

\newcommand{\GL}{{\text{GL}}}
\newcommand{\Or}{{\text{O}}}

\newcommand{\Sp}{{\text{Sp}}}

\newcommand{\inv}{^{-1}}

\newcommand{\lexp}[2]{\kern\scriptspace\vphantom{#2}^{#1}\kern-\scriptspace#2}

\newcommand\be{\begin{equation}}
\newcommand\ee{\end{equation}}
\newcommand\ot{\otimes}
\newcommand\ve{\varepsilon}
\newcommand\lr{{\longrightarrow\;}}

\newcommand\half{{[\frac{n+1}{2}]}}

\begin{document}

\title [Second fundamental theorem]{The second fundamental theorem of \\ invariant theory for the orthogonal group}

\author{Gustav~Lehrer and Ruibin~Zhang}
\address{School of Mathematics and Statistics,
University of Sydney, NSW 2006, Australia.}
\email{gustav.lehrer@sydney.edu.au, ruibin.zhang@sydney.edu.au}
\begin{abstract}Let $V=\C^n$ be endowed with an orthogonal form and $G=\Or(V)$
be the corresponding orthogonal group. Brauer showed in 1937 that there is
a surjective homomorphism $\nu:B_r(n)\to\End_G(V^{\otimes r})$,
where $B_r(n)$ is the $r$-string Brauer
algebra with parameter $n$. However the kernel of
$\nu$ has remained elusive. In this paper we show that,
in analogy with the case of $\GL(V)$, for $r\geq n+1$, $\nu$ has kernel
which is generated by a single idempotent element $E$, and we give a simple explicit formula for $E$.
Using the theory of cellular algebras, we show how $E$ may be used to determine the
multiplicities of the irreducible representations of $\Or(V)$ in $V^{\ot r}$.
We also show how our results extend to the case where $\C$ is replaced by
an appropriate field of positive characteristic, and comment on quantum analogues of our results.
\end{abstract}
\maketitle

\section{Introduction}

Let $K$ be a field of characteristic zero, and let $V=K^n$ be an
$n$-dimensional vector space
with a non-degenerate symmetric bilinear form $(-,-)$, and assume that with respect to some basis
of $V$, the form has matrix equal to the identity matrix. Equivalently, there is a
basis $\{b_1,\dots,b_n\}$ such that $(b_i,b_j)=\delta_{ij}$; such a basis is called orthonormal.
The {\em orthogonal group} $\Or(V)$ is the isometry group of this form, defined as
$\Or(V)=\{g\in\GL(V)\mid (gv,gw)=(v,w)\;\forall\; v,w\in V\}$. In \cite{Br}, Brauer showed that
the first fundamental theorem of invariant theory for $\Or(V)$ implies that there is a
surjective map $\nu$ from the Brauer algebra $B_r(n)$ over $K$ to $\End_{\Or(V)}(V^{\otimes r})$, but
the fact \cite{DHW, RS} that $B_r(n)$ is semisimple if and only if $r\leq n+1$, has complicated the determination
of the kernel of $\nu$, and therefore limited the use of this fact.

In this work we determine $\ker(\nu)$. More specifically, we show that $\ker(\nu)$ is
generated as an ideal of $B_r(n)$ by a single
idempotent $E$, which we describe explicitly. Using the fact that $B_r(n)$ has a cellular structure
\cite{GL96}, we show how this fact may be used to illuminate the Schur-Weyl duality
between the actions of $\Or(V)$ and $B_r(n)$ on $V^{\otimes r}$, by using $E$ to describe the
radicals of the canonical forms on the relevant cell modules of $B_r(n)$.

The special case $n=3$ has been treated in \cite{LZ2,LZ3}, as has its quantum analogue
for the Birman-Murukami-Wenzl (BMW) algebra \cite{BW}. This latter work was done in the context of
the $3$-dimensional irreducible representation of $\fsl_2$. The case
of the symplectic group $\Sp_{2n}(K)$ and its quantum analogue, which seems rather
different from the present case, has been treated by Hu and Xiang in \cite{HX}.

\protect\section{The Brauer algebra}

\protect\subsection{Generalities}

Let $K$ be a field of characteristic zero and
let $\delta\in K$. For any positive integer $r$, the Brauer algebra
$B_r(\delta)$ \cite{Br} is the $K$-algebra with basis the set of diagrams with $2r$ nodes,
or vertices, labelled as in Figure \ref{fig3},
in which each node is joined to just one other one.

\begin{figure}[hb]
\begin{center}
\begin{tikzpicture}[scale=1]

%% draw filled dots at required points at bottom
\foreach \x in {1,2,3,6,7}
\filldraw(\x,0) circle (0.1cm);
%% draw filled dots at required points on top
\foreach \x in {1,2,3,6,7}
\filldraw(\x,2) circle (0.1cm);
%% draw straight lines where required
%\foreach \x in {1,4,6,9,11,14}
%\draw (\x,3)--(\x,0);

%% label the points in top row
\draw node[above] at (1,2){1};\draw node[above] at (2,2){2};\draw node[above] at (3,2){3};
\draw node[above] at (6,2){r-1};\draw node[above] at (7,2){r};

%% label the points in bottom row
\draw node[below] at (1,0){r+1};\draw node[below] at (2,0){r+2};\draw node[below] at (3,0){r+3};
\draw node[below] at (6,0){2r-1};\draw node[below] at (7,0){2r};

%% draw curved lines where required, top
%\draw(5,3) .. controls (6,2.2) and (9,2.2) .. (10,3);
%\draw(7,3) .. controls (7.5,2) and (9.5,2) .. (10,3);
%\draw(6,3) .. controls (7,1.5) and (10,1.5) .. (11,3);

%% draw curved lines where required, bottom
%\draw(5,0) .. controls (6,0.8) and (9,0.8) .. (10,0);
%\draw(7,0) .. controls (7.5,1) and (9.5,1) .. (10,0);
%\draw(6,0) .. controls (7,1.5) and (10,1.5) .. (11,0);

%% draw ... between straight lines
\draw(4.5,0) node {\large{$\cdots$}};
\draw(4.5,2) node {\large{$\cdots$}};
%\draw(12.5,1.5) node {\large{$\cdots$}};

\end{tikzpicture}
%\centerline{Figure 1}
\end{center}
\caption{ }
\label{fig3}
\end{figure}

Note that each diagram in $B_r(\delta)$ may be thought of as a graph with
vertices $\{1,\dots,2r\}$ in which each vertex is joined to just one other one.
We will often refer to the `edges' of such a diagram, and speak of `horizontal edges'
and `vertical edges' (the latter also known as `through strings') as respectively
those joining vertices in the same row, or in different rows.

The composite $D_1\circ D_2$ of two diagrams $D_1$ and $D_2$ is
obtained by concatenation
of diagrams, placing $D_1$ above $D_2$, with the intermediate nodes and any free loops being erased.
The product $D_1D_2$ in $B_r(\delta)$ is $\delta^{l(D_1,D_2)}D_1\circ D_2$, where $l(D_1,D_2)$ is
the number of deleted free loops.

We shall need to consider certain special elements of $B_r(\delta)$, which
we now describe.

For $i=1,\dots,r-1$, $s_i$ is the diagram shown in Figure \ref{fig4}.

\begin{figure}[hb]
\begin{center}
\begin{tikzpicture}[scale=1]

%% draw filled dots at required points at bottom
\foreach \x in {1,4,5,6,7,10}
\filldraw(\x,0) circle (0.1cm);
%% draw filled dots at required points on top
\foreach \x in {1,4,5,6,7,10}
\filldraw(\x,2) circle (0.1cm);
%% draw straight lines where required
\foreach \x in {1,4,7,10}
\draw (\x,2)--(\x,0);

%% draw slanted lines where required
\draw (5,2)--(6,0);
\draw (5,0)--(6,2);

%% label the points in top row
\draw node[above] at (1,2){1};\draw node[above] at (5,2){i};
\draw node[above] at (6,2){i+1};\draw node[above] at (10,2){r};

%% label the points in bottom row
\draw node[below] at (1,0){r+1};\draw node[below] at (5,0){r+i};
\draw node[below] at (6,0){r+i+1};\draw node[below] at (10,0){2r};

%% draw curved lines where required, top
%\draw(5,3) .. controls (6,2.2) and (9,2.2) .. (10,3);
%\draw(7,3) .. controls (7.5,2) and (9.5,2) .. (10,3);
%\draw(6,3) .. controls (7,1.5) and (10,1.5) .. (11,3);

%% draw curved lines where required, bottom
%\draw(5,0) .. controls (6,0.8) and (9,0.8) .. (10,0);
%\draw(7,0) .. controls (7.5,1) and (9.5,1) .. (10,0);
%\draw(6,0) .. controls (7,1.5) and (10,1.5) .. (11,0);

%% draw ... between straight lines
\draw(2.5,1) node {\large{$\cdots$}};
\draw(8.5,1) node {\large{$\cdots$}};
%\draw(12.5,1.5) node {\large{$\cdots$}};

\end{tikzpicture}
%\centerline{Figure 2}
\end{center}
\caption{ }
\label{fig4}
\end{figure}

For each pair $i,j$ with $1\leq i<j\leq r$ define the diagram $e_{i,j}$
as depicted in Figure \ref{fig5}.

\begin{figure}[hb]
\begin{center}
\begin{tikzpicture}[scale=1]

%% draw filled dots at required points at bottom
\foreach \x in {1,4,5,6,9,10,11,14}
\filldraw(\x,0) circle (0.1cm);
%% draw filled dots at required points on top
\foreach \x in {1,4,5,6,9,10,11,14}
\filldraw(\x,3) circle (0.1cm);
%% draw straight lines where required
\foreach \x in {1,4,6,9,11,14}
\draw (\x,3)--(\x,0);

%% label the points in top row
\draw node[above] at (1,3){1};\draw node[above] at (4,3){i-1};\draw node[above] at (5,3){i};
\draw node[above] at (6,3){i+1};\draw node[above] at (9,3){j-1};\draw node[above] at (10,3){j};
\draw node[above] at (11,3){j+1};\draw node[above] at (14,3){r};

%% label the points in bottom row
\draw node[below] at (1,0){r+1};\draw node[below] at (5,0){r+i};\draw node[below] at (10,0){r+j};
\draw node[below] at (14,0){2r};

%% draw curved lines where required, top
\draw(5,3) .. controls (6,2.2) and (9,2.2) .. (10,3);
%\draw(7,3) .. controls (7.5,2) and (9.5,2) .. (10,3);
%\draw(6,3) .. controls (7,1.5) and (10,1.5) .. (11,3);

%% draw curved lines where required, bottom
\draw(5,0) .. controls (6,0.8) and (9,0.8) .. (10,0);
%\draw(7,0) .. controls (7.5,1) and (9.5,1) .. (10,0);
%\draw(6,0) .. controls (7,1.5) and (10,1.5) .. (11,0);

%% draw ... between straight lines
\draw(2.5,1.5) node {\large{$\cdots$}};
\draw(7.5,1.5) node {\large{$\cdots$}};
\draw(12.5,1.5) node {\large{$\cdots$}};

\end{tikzpicture}
%\centerline{Figure 3}
\end{center}
\caption{ }
\label{fig5}
\end{figure}

The following facts are all well known.
\begin{lemma}\label{lem:brprops}
\begin{enumerate}
\item The elements $s_1,\dots,s_{r-1}$ generate a subalgebra of $B_r(\delta)$, isomorphic to
the group algebra $K\Sym_r$ of the symmetric group $\Sym_r$.

\item The elements $e_{i,j}$ satisfy $e_{i,j}^2=\delta e_{i,j}$, and if
$i,j,k$ and $\ell$ are distinct, $e_{i,j}$ commutes
with $e_{k,\ell}$.

\item If we write $e_i=e_{i,i+1}$ for $i=1,\dots,r-1$,
$B_r(\delta)$ has a presentation as $K$-algebra with generators
$\{s_1,\dots,s_{r-1};e_1,\dots,e_{r-1}\}$, and relations
$s_i^2=1,\; e_i^2=\delta e_i,\; s_ie_i=e_is_i=e_i$ for all $i$,
$s_is_j=s_js_i,\;s_ie_j=e_js_i,\;e_ie_j=e_je_i$ if $|i-j|\geq 2$,
and $s_is_{i+1}s_i=s_{i+1}s_is_{i+1},\; e_ie_{i\pm 1}e_i=e_i$ and
$s_ie_{i+1}e_i=s_{i+1}e_i,\; e_{i+1}e_is_{i+1}=e_{i+1}s_{i}$
and $e_is_{i\pm 1}e_i=e_i$ for all applicable $i$.
\end{enumerate}
\end{lemma}

\subsection{Some special notation}\label{ss:alt}

For positive integers $k,l$ with $k\leq l$ define $[k,l]:=\{k,k+1,k+2,\dots,l\}$.
For any subset $S\subseteq[1,r]$, $\Sym(S)$ is the subgroup of $\Sym_r$ which fixes
each element of $[1,r]\setminus S$. For any subgroup $H\leq\Sym_r$, define
$a(H)=\sum_{h\in H}\ve(h)h\in B_r(\delta)$, where $\ve$ is the alternating character of
$\Sym_r$. This is referred to as the `alternating element'
of $KH$. The following elementary observation is well known but very useful.

\begin{lemma}\label{lem:capzero}
Suppose the subgroup $H\leq\Sym_r$ contains the simple transposition $s_{ij}$
which interchanges $i$ and $j$.
If $e_{i,j}$ is the element defined above, then $e_{i,j}a(H)=a(H)e_{i,j}=0$.
\end{lemma}
\begin{proof}
Since $s_{ij}a(H)=-a(H)=a(H)s_{ij}$, we have $a(H)=\frac{1}{2}(1-s_{ij})a(H)
=\frac{1}{2}a(H)(1-s_{ij})$. But $e_{i,j}s_{ij}=s_{ij}e_{i,j}=e_{i,j}$, and the
result is clear.
\end{proof}

\section{The fundamental theorems of invariant theory for $O(n)$}

\subsection{First (linear) formulation}

For any positive integer $t$ the space $V^{\otimes t}$ (also denoted $T^t(V)$)
is an $\Or(V)$-module in the usual way: $g(v_1\ot\dots\ot v_t)=gv_1\ot gv_2\ot\dots\ot gv_t$.
Moreover the given form on $V$ provides a non-degenerate symmetric bilinear form $[-,-]$
on $V^{\ot t}$, given by $[v_1\ot\dots\ot v_t,w_1\ot\dots\ot w_t]:=\prod_{i=1}^t(v_i,w_i)$,
which permits the identification of $V^{\ot t}$ with its dual space ${V^{\ot t}}^*$.

The space of invariants $({{V^{\ot t}}^*})^{\Or(V)}$ is defined as the space of linear functions
which are constant on $\Or(V)$ orbits.
One formulation of the first fundamental theorem of invariant theory for $\Or(V)$ \cite{W, ABP}
is as follows (see \cite{Lo,LZZ,P1,P2}, \cite[Proposition 21]{Ri}).

\begin{theorem}\label{thm:orthfft} The space $({{V^{\ot t}}^*})^{\Or(V)}$ is zero if $t$ is odd.
If $t=2r$ is even, then
any element of $({{V^{\ot t}}^*})^{\Or(V)}$ is a linear combination of maps of the form
$$
\gamma_D:v_1\ot\dots\ot v_{2r}\mapsto \prod_{(i,j)\text{ is an edge of $D$}}(v_{i},v_j),
$$
where $D$ is a diagram in $B_r(n)$.
\end{theorem}

The second fundamental theorem provides a description of all linear relations among these
functions $\gamma_D$. Let us begin by describing some obvious linear
relations. Suppose $r\geq n+1$ (recall that $\dim V=n$).

Let $S$ and $S'$ be disjoint subsets of $[1,2r]$ such that
$|S|=|S'|=n+1$ and $S\cap S'=\emptyset$, and let $\beta$ be any pairing
of the vertices $\{1,\dots,2r\}\setminus (S\amalg S')$.

\begin{definition}\label{def:dpi}
For $\pi\in\Sym_{n+1}$, $S=\{i_1,\dots,i_{n+1}\}$, $S'=\{j_1,\dots,j_{n+1}\}$,
let $D_\pi(S,S',\beta)$ be the Brauer diagram with
edges $\{(i_k,j_{\pi(k)})\mid k=1,2,\dots,n+1\}\amalg\beta$, and denote by
$\gamma_{D_\pi(S,S',\beta)}$ the corresponding linear functional on $V^{\ot 2r}$
as above.

Define $\gamma(S,S',\beta):=\sum_{\pi\in\Sym_{n+1}}\ve(\pi)\gamma_{D_\pi(S,S',\beta)}$.
\end{definition}

The next statement describes some obvious linear relations among the $\gamma_D$.

\begin{lemma}\label{lem:gammainker}
We have, for each $S,S',\beta$ as above, $\gamma(S,S',\beta)=0(\in(V^{\ot 2r})^*$.
\end{lemma}
\begin{proof}
If $S=i_1<i_2<\dots<i_{n+1}$ and $S'=j_1<j_2<\dots<j_{n+1}$,
then for any $v_1\ot\dots\ot v_{2r}\in V^{\ot 2r}$,
clearly the $(n+1)\times(n+1)$ matrix with $k,l$ entry $(v_{i_k},v_{j_l})$ is singular,
since the rows are linearly dependent, as by dimension, there is a linear relation among the
$v_{i_k}$. The lemma follows by observing that $\gamma(S,S',\beta)$ is a multiple of the
function $v_1\ot\dots\ot v_{2r}\mapsto \det(v_{i_k},v_{j_l})$, which is zero.
\end{proof}

The second fundamental theorem for $\Or(V)$ may be stated as follows \cite[Prop. 21]{Ri}.

\begin{theorem}\label{thm:orthsft} If $r\leq n$, the $\gamma_D$ form a basis
of the space of $\Or(V)$-invariants on $(V^{\ot 2r})^*$. If $r\geq n+1$, then
any linear relation among the functionals $\gamma_D$ is a linear consequence
of the relations in Lemma \ref{lem:gammainker}.
\end{theorem}

\subsection{Second formulation}

Our objective is to reinterpret the first and second
fundamental theorems in terms of the Brauer algebra $B_r(n)$, which is the algebra
described above, with $\delta$ replaced by $n=\dim V$. For this purpose we consider
some maps which we now define. There is a canonical map $\xi:V\ot V\to \End(V)$ given by
$\xi(v\ot w):x\mapsto (w,x)v$ (where $v,w,x\in V$).
Define $A:V^{\ot 2r}\to \End(V^{\ot r})(\simeq (\End(V))^{\ot r})$
by
$$
A(v_1\ot\dots\ot v_{2r})=\xi(v_1\ot v_{r+1})\ot\xi(v_2\ot v_{r+2})\ot\dots\ot
\xi(v_r\ot v_{2r}).
$$
This map respects the action of $\Or(V)$ on its domain and range,
where $\Or(V)$ acts on $\End(V)^{\ot r}$ by conjugation: for $g\in \Or(V)$ and
$\alpha_1\ot\dots\ot\alpha_r\in\End(V)^{\ot r}$, $g\cdot(\alpha_1\ot\dots\ot\alpha_r):=
g\alpha_1g\inv\ot\dots\ot g\alpha_rg\inv$.

Next observe that if $b_1,\dots,b_n$ is an orthonormal basis of $V$, the element
$\gamma_0:=\sum_{i=1}^nb_i\ot b_i$ is $\Or(V)$-invariant, and independent of the basis.
Hence the linear map $\phi:V\ot V\to V\ot V$ defined by $\phi(v\ot w)=(v,w)\gamma_0$
commutes with $\Or(V)$, i.e. $\phi\in\End_{\Or(V)}V^{\ot 2}$. This map $\phi$
is called the contraction map. For $i=1,2,\dots r-1$, define $\phi_i\in\End(V^{\ot r})$
as the endomorphism which is the contraction $\phi$ on the tensor product of the
$i^{\text th}$ and $(i+1)^{\text st}$ factors, and is the identity on all other factors.
It was proved by Brauer \cite{Br} that if $\Sym_r\subset B_r(n)$ acts via place permutations
on $V^{\ot r}$, i.e. if we define, for $\sigma\in\Sym_r$, $\nu(\sigma)\cdot (v_1\ot\dots\ot v_r)
=v_{\sigma\inv(1)}\ot\dots\ot v_{\sigma\inv(r)}$, then $\nu$ extends to a homomorphism
of associative algebras
$$
\nu: B_r(n)\to \End_{\Or(V)}(V^{\ot r})\subseteq \End(V^{\ot r})
$$
by defining $\nu(e_i)=\phi_i$.
Note that if $\sigma\in\Sym_r$ the diagram representing $\sigma$ in $B_r(n)$
has edges $(\sigma(i),r+i)$, for $i=1,2,\dots,r$

Finally, we define a linear map $\tau:B_r(n)\to V^{\ot 2r}$ as follows. For any diagram
$D\in B_r(n)$, define $\tau(D)=t_D$, where
$$
t_D=\sum_{\substack{i_1,\ldots,i_{2r}=1\\i_j=i_k\;
\text{if $(j,k)$ is an edge of $D$}}}^n b_{i_1}\otimes \cdots\otimes b_{i_{2r}},
$$
where $b_1,\dots,b_n$ is an orthonormal basis of $V$. Notice that $t_D$ is the unique element
of $V^{\ot 2r}$ such that in the notation above,
$$
[-,t_D]=\gamma_D.
$$
Now consider the following diagram of linear maps.

\be\label{diag:triangle}
\psmatrix[colsep=14mm,rowsep=10mm]
%K\Sym_r & T^r(V)\ot T^r(V^*)\cong(T^r(V)\ot T^r(V^*))^*\\
 & B_r(n)\\
 V^{\ot 2r} & &\End(V^{\ot r})
\psset{nodesep=2pt}\everypsbox{\scriptstyle}
\ncline{->}{1,2}{2,1}\tlput{\tau}
%  \ncline%[linestyle=dashed]
\ncline{->}{1,2}{2,3}\trput{\nu}%^{\hskip -.2cm\mu}
\ncline{->}{2,1}{2,3}^A
\endpsmatrix
\ee
The next statement is crucial for understanding the second fundamental theorem in the context of the
Brauer algebra.
\begin{proposition}\label{prop:commute}
The diagram \eqref{diag:triangle} commutes.
\end{proposition}
\begin{proof}
We start by observing that the group $\Sym_r\times\Sym_r$ acts on each of the
three spaces in the diagram
as follows. Let $(\sigma_1,\sigma_2)\in\Sym_r\times\Sym_r$.
Then for a diagram $D\in B_r(n)$,
$(\sigma_1,\sigma_2)\cdot D:=\sigma_1 D\sigma_2\inv$; for
$v_1\ot\dots\ot v_r\ot w_1\ot\dots\ot w_r\in V^{\ot 2r}$,
$(\sigma_1,\sigma_2)\cdot(v_1\ot\dots\ot v_r\ot w_1\ot\dots\ot w_r):=
\nu(\sigma_1)(v_1\ot\dots\ot v_r)\ot
\nu(\sigma_2)(w_1\ot\dots\ot w_r)$, and for $T\in\End(V^{\ot r})$,
$(\sigma_1,\sigma_2).T:=\nu(\sigma_1)T\nu(\sigma_2\inv)$.

Next, a straightforward computation shows that each of the maps
$\nu, \tau$ and $A$ respects the action
of $\Sym_r\times\Sym_r$.

Now each diagram $D\in B_r(n)$ may be written in the form $D=\sigma_1 l(s) \sigma_2\inv$ for some
$(\sigma_1,\sigma_2)\in \Sym_r\times\Sym_r$ and some diagram $l(s)$, where
$s\in\{1,2,\dots,[\frac{n+1}{2}]\}$,
and the diagram $l(s)$ is as shown in Figure \ref{fig6}.

\begin{figure}[hb]
\begin{center}
\begin{tikzpicture}[scale=1]

%% draw filled dots at required points at bottom
\foreach \x in {1,2,3,4,7,8,9,10,13}
\filldraw(\x,0) circle (0.1cm);
%% draw filled dots at required points on top
\foreach \x in {1,2,3,4,7,8,9,10,13}
\filldraw(\x,3) circle (0.1cm);
%% draw straight lines where required
\foreach \x in {9,10,13}
\draw (\x,3)--(\x,0);

%% label the points in top row
\draw node[above] at (1,3){1};\draw node[above] at (2,3){2};\draw node[above] at (3,3){3};
\draw node[above] at (4,3){4};\draw node[above] at (7,3){2s-1};\draw node[above] at (8,3){2s};
\draw node[above] at (9,3){2s+1};\draw node[above] at (10,3){2s+2};\draw node[above] at (13,3){r};

%% label the points in bottom row
\draw node[below] at (1,0){r+1};\draw node[below] at (2,0){r+2};\draw node[below] at (9,0){2s+r+1};
\draw node[below] at (13,0){2r};

%% draw curved lines where required, top
\draw(1,3) .. controls (1.5,2)  .. (2,3);
\draw(3,3) .. controls (3.5,2)  .. (4,3);
\draw(7,3) .. controls (7.5,2)  .. (8,3);
%\draw(7,3) .. controls (7.5,2) and (9.5,2) .. (10,3);
%\draw(6,3) .. controls (7,1.5) and (10,1.5) .. (11,3);

%% draw curved lines where required, bottom
\draw(1,0) .. controls (1.5,1)  .. (2,0);
\draw(3,0) .. controls (3.5,1)  .. (4,0);
\draw(7,0) .. controls (7.5,1)  .. (8,0);
%\draw(5,0) .. controls (6,0.8) and (9,0.8) .. (10,0);
%\draw(7,0) .. controls (7.5,1) and (9.5,1) .. (10,0);
%\draw(6,0) .. controls (7,1.5) and (10,1.5) .. (11,0);

%% draw ... between straight lines
\draw(5.5,1.5) node {\large{$\cdots$}};
\draw(11.5,1.5) node {\large{$\cdots$}};
%\draw(12.5,1.5) node {\large{$\cdots$}};

\end{tikzpicture}
%\centerline{Figure 4}
\end{center}
\caption{ }
\label{fig6}
\end{figure}

%\begin{topinsert}
%\botcaption{Figure 6}
%\endcaption
%\end{insert}

Hence if we were able to prove that for each $s$, we have
\be\label{eq:reduction}
A\tau(l(s))=\nu(l(s)),
\ee
we would have, for $(\sigma_1,\sigma_2)\in\Sym_r\times\Sym_r$,
$A\tau(\sigma_1 l(s)\sigma_2\inv)=(\sigma_1,\sigma_2)\cdot A\tau(l(s))
=(\sigma_1,\sigma_2)\cdot\nu(l(s))=\nu(\sigma_1 l(s)\sigma_2\inv)$, and the
Proposition would follow. Hence it remains to prove \eqref{eq:reduction},
and this may be checked directly, given the identities
$$
\sum_{i=1}^n\xi(b_i\ot b_i)=\Id_V
$$
and
$$
\sum_{i,j=1}^n\xi(b_i\ot b_j)\ot\xi(b_i\ot b_j)=\phi:V\ot V\to V\ot V,
$$
where $\phi$ is the contraction introduced above.
\end{proof}

\begin{corollary}\label{cor:fftbrauer}
The map $\nu$ maps $B_r(n)$ surjectively to $\End_{\Or(V)}(V^{\ot r})$,
and $\ker (\nu)=\ker(\tau)$.
\end{corollary}
\begin{proof} Since $A$ is an $\Or(V)$-equivariant isomorphism, it restricts to an
isomorphism from the $\Or(V)$-invariants of $V^{\ot 2r}$ to those of $\End(V^{\ot r})$.
It follows from Proposition \ref{prop:commute} that since $\tau$ has image
$(V^{\ot 2r})^{\Or(V)}$,
$\im(\nu)=\End(V^{\ot r})^{\Or(V)}=\End_{\Or(V)}(V^{\ot r})$.
The fact that $\ker(\nu)=\ker(\tau)$ is clear from the commutativity of the diagram.
\end{proof}

\section{Statement of the main result}

\subsection{Overview}

We shall translate Theorem \ref{thm:orthsft}, which uses only the linear
structure, into an explicit description of the ideal $\ker(\nu)$ above.
We start with the following easy observation.
\begin{lemma}\label{lem:ker1} For each triple $(S,S',\beta)$ as in Definition
\ref{def:dpi}, define the element
$$
b(S,S',\beta)=\sum_{\pi\in\Sym_{n+1}}\ve(\pi)D_\pi(S,S',\beta)\in B_r(n).
$$
Then the elements $b(S,S',\beta)$ span $\ker(\nu)$.
\end{lemma}
\begin{proof}
Writing $t(S,S',\beta)=\tau(b(S,S',\beta))$, it is clear that the functional
$$
x\mapsto [x,t(S,S',\beta)]\;\;(x\in V^{\otimes r})
$$
on $V^{\ot 2r}$
is equal to $\gamma(S,S',\beta)$, which is
zero by Lemma \ref{lem:gammainker}. Hence $b(S,S',\beta)\in\ker(\tau)=\ker(\nu)$.
It follows from Theorem \ref{thm:orthsft} that these elements span the kernel.
\end{proof}

Next, we identify a small subset of the elements $b(S,S',\beta)$
of $B_r(n)$, which
are such that the ideal of $B_r(n)$ which they generate contains, for each triple $S,S',\beta$
as in Definition \ref{def:dpi}, the element $b(S,S',\beta)$.
By Lemma \ref{lem:ker1} the ideal they generate is the whole kernel. We then show that this ideal
is in fact generated by one of those elements.

\subsection{Formulation}

We begin by defining certain elements of $B_r(n)$.
For this purpose, the following
notation will be convenient. If $k,l$ are integers such that $1\leq k<l$, write
$a(k,l):=a(\Sym\{k,k+1,\dots,l\})$ (see \S\ref{ss:alt}). By convention, if $k\geq l$, $a(k,l)=1$.
\begin{definition}\label{def:br-elts} For $i=0,1,\dots,\half$ define the
following elements of $B_r(n)$
\begin{enumerate}
\item $F_i:=a(1,i)a(i+1,n+1)$, where $F_0$ is interpreted as $a(1,n+1)$.
\item For $j=0,1,2,\dots,i$, define $e_i(j)=e_{i,i+1}e_{i-1,i+2}\dots e_{i-j+1,i+j}$.
This is a diagram
with $j$ top (resp. bottom) horizontal arcs. Note that $e_i(0)=1$ by convention.
\item Using the above notation, define elements $E_i$ $(i=0,1,\dots,\half)$ as follows.
$$
E_i=\sum_{j=0}^i(-1)^jc_i(j)F_ie_i(j)F_i,
$$
where $c_i(j)=\left((i-j)!(n+1-i-j)!(j!)^2 \right)\inv$.
\end{enumerate}
\end{definition}

Note that the leading term ($j=0$) of $E_i$ is $(i!(n+1-i)!)\inv F_i^2=F_i$.

Our main result is
\begin{theorem}\label{thm:main}
In the notation of Definition \ref{def:br-elts}, write $E=E_\half$.
Then $E^2=(\half)!(n+1-\half)!E$.
If $r\leq n$ the map $\nu:B_r(n)\to \End_{\Or(V)}(V^{\ot r})$ is an isomorphism.
If $r\geq n+1$, $\ker(\nu)$ is generated as an ideal of $B_r(n)$ by $E$.
\end{theorem}

We begin with the following result, whose proof will require arguments involving
the geometry of Brauer diagrams.

\begin{proposition}\label{prop:einker}
Assume $r\geq n+1$. For each $i=0,1,\dots,\half$, $E_i\in\ker(\nu)$.
\end{proposition}

\begin{proof}
We shall show that each element $E_i$ is of the form
$b(S,S',\beta)$ for some triple $S,S',\beta$.

For this, let $S_i=\{1,2,\dots,i,i+1+r,i+2+r,\dots,n+1+r\}$,
and $S_i'=\{i+1,i+2,\dots,n+1,r+1,r+2,\dots, r+i\}$. Then $|S_i|=|S'_i|=n+1$
and $S_i\cap S'_i=\emptyset$. In Figure \ref{fig8}, the points of $S_i$ are denoted
by $\circ$, those of $S'_i$ by $*$ and the others by $\bullet$.

\begin{figure}[hb]
\begin{center}
\begin{tikzpicture}[scale=1]

%% draw filled dots at required points at bottom
\foreach \x in {1,2,5}
\draw node at (\x,0){*};
\foreach \x in {6,7,10}
\draw(\x,0) circle (0.1cm);
\foreach \x in {11,14}
\filldraw(\x,0) circle (0.1cm);
%% draw filled dots at required points on top
\foreach \x in {1,2,5}
\draw(\x,2) circle (0.1cm);
\foreach \x in {6,7,10}
\draw node at (\x,2){*};
\foreach \x in {11,14}
\filldraw(\x,2) circle (0.1cm);
%% draw straight lines where required
\foreach \x in {11,14}
\draw (\x,2)--(\x,0);

%% label the points in top row
\draw node[above] at (1,2){1};\draw node[above] at (2,2){2};\draw node[above] at (5,2){i};
\draw node[above] at (6,2){i+1};\draw node[above] at (10,2){n+1};
\draw node[above] at (11,2){n+2};\draw node[above] at (14,2){r};

%% label the points in bottom row
\draw node[below] at (1,0){r+1};\draw node[below] at (2,0){r+2};\draw node[below] at (5,0){r+i};
\draw node[below] at (6,0){i+1+r};\draw node[below] at (10,0){n+1+r};
\draw node[below] at (14,0){2r};

%% draw curved lines where required, top
%\draw(5,3) .. controls (6,2.2) and (9,2.2) .. (10,3);
%\draw(7,3) .. controls (7.5,2) and (9.5,2) .. (10,3);
%\draw(6,3) .. controls (7,1.5) and (10,1.5) .. (11,3);

%% draw curved lines where required, bottom
%\draw(5,0) .. controls (6,0.8) and (9,0.8) .. (10,0);
%\draw(7,0) .. controls (7.5,1) and (9.5,1) .. (10,0);
%\draw(6,0) .. controls (7,1.5) and (10,1.5) .. (11,0);

%% draw ... between straight lines
\draw(3.5,0) node {\large{$\cdots$}};\draw(8.5,0) node {\large{$\cdots$}};
\draw(12.5,0) node {\large{$\cdots$}};
\draw(3.5,2) node {\large{$\cdots$}};\draw(8.5,2) node {\large{$\cdots$}};
\draw(12.5,2) node {\large{$\cdots$}};
%\draw(12.5,1.5) node {\large{$\cdots$}};

\end{tikzpicture}
%\centerline{Figure 5}
\end{center}
\caption{ }
\label{fig8}
\end{figure}

With $S,S'$ as above, $\{1,\dots,2r\}\setminus(S_i\amalg S_i')=
\{n+2,n+3,\dots,r,n+2+r,n+3+r,\dots,2r\}$, and we take $\beta_i$ to be the pairing
$(n+2,n+2+r),\dots,(r,2r)$.
We shall show that
\be\label{pf:3}
E_i=b(S_i,S_i',\beta_i).
\ee

Now $b(S_i,S_i',\beta_i)$ is the alternating sum of the set $\CS_i$
of $(n+1)!$ diagrams of the form Figure \ref{fig8} above,
in which each point of $S_i$ is joined to a point of $S_i'$. Let $H=\Sym\{1,\dots,i\}
\times\Sym\{i+1,\dots,n+1\}$, regarded as a subgroup of the algebra $B_r(n)$. Clearly
$H\times H$ acts on this set $\CS_i$ of diagrams, via $(h,h').D=hD{h'}\inv$,
preserving the number of horizontal arcs.
Moreover each diagram in $\CS_i$ may be transformed by $H\times H$ into a unique diagram
$e_i(j)$ for some $j=0,1,\dots,i$. That is, $H\times H$ has $i+1$ orbits on $\CS_i$,
and the $e_i(j)$ form a set of orbit representatives. It is therefore clear that
the alternating sum of the diagrams in the orbit of $e_i(j)$ is a scalar times
$F_ie_i(j)F_i$. If the trivial diagram (in the orbit of $e_i(0)=1$) has sign
$+1$, then observing that $e_i(j+1)$ is obtained from $e_i(j)$ by a simple interchange in $S_i'$
(of $r+i-j$ with $i+j+1$), we see that $b(S_i,S_i',\beta_i)=\sum_j(-1)^jc_i(j)F_ie_i(j)F_i$, where
$c_i(j)$ is the inverse of $|\{(h,h')\in H\times H\mid he_i(j)h'=e_i(j)\}|$. This proves
\eqref{pf:3}, and we are done.
\end{proof}

\section{Some computations in the Brauer algebra}

In this section, we carry out some necessary
computations in the Brauer algebra $B_r(\delta)$, where $\delta$ is arbitrary,
and apply them to prove a key annihilation result (Theorem \ref{thm:ann} below).

\subsection{Arcs in the Brauer algebra}\label{ss:basic}

\begin{lemma}\label{lem:comp1}
(i) In the group algebra $K\Sym_r$, we have, in the notation of \S\ref{ss:alt},
\[
a(\Sym_n)=a(\Sym_{n-1})-|\Sym_{n-2}|\inv a(\Sym_{n-1})s_{n-1}a(\Sym_{n-1}).
\]
(ii) In the $K$-algebra $B_r(\delta)$, let $F=a(1,i)a(i+1,s)$ and
$F'=a(1,i-1)a(i+2,s)$. Then for $2\leq i\leq s-2$,
\[
\begin{aligned}
e_{i,i+1}Fe_{i,i+1}=&(\delta-s+2)F'e_{i,i+1}+
[(i-2)!(s-i-2)!]\inv e_{i,i+1}F'e_{i-1,i+2}F'.
\end{aligned}
\]
(iii) The statement (ii) above remains true if $i=1$, provided that $e_{0,3}$
is interpreted as $0$, and $a(k,l)=1$ whenever $k\geq l$.
\end{lemma}
\begin{proof}
The first statement is a simple consequence of the double coset decomposition
$\Sym_n=\Sym_{n-1}\amalg \Sym_{n-1}s_{n-1}\Sym_{n-1}.$

For the second, observe that from (i) we have $a(1,i)=a(1,i-1)-(i-2)!\inv a(1,i-1)s_{i-1} a(1,i-1)$,
and $a(i+1,s)=a(i+2,s)-(s-i-2)!\inv a(i+2,s)s_{i+1}a(i+2,s)$.
One now computes directly, using the relations in $B_r(\delta)$,
the key relation here being $e_is_{i\pm 1}e_i=e_i$. The third statement,
concerning the case $i=1$, follows from the above argument, but may also be
computed directly.
\end{proof}

The computation above may be usefully iterated as follows.

\begin{corollary}\label{cor:basic}
Assume that $0\leq i\leq s-i$, and that $0\leq j\leq i-1$.
For $k=0,1,\dots,i$, write $J_k=a(1,i-k)a(i+k+1,s)$, so that $J_0=F$
in Lemma \ref{lem:comp1},
$J_{i-1}=a(2i,s)$ and $J_i=a(2i+1,s)$, interpreted as $1$ if $i=2s$.
Write $e(j)=e_{i-j+1,i+j}$ for $j=0,1,\dots,i$; by convention
$e(j)=0$ for $j>i$, and we note that $e(j)J_k=J_ke(j)$ for $j\leq k$.
Then\newline
(i) We have, for all $i,s$ as above and for $j$ such that $0\leq j\leq i-1$,
\be\label{eq:eje}
\begin{aligned}
e(j+1)J_j&e(j+1)=(\delta-s+2j+2)J_{j+1}e(j+1)+\\
&((i-j-2)!(s-i-j-2)!)\inv e(j+1)J_{j+1}e(j+2)J_{j+1}.\\
\end{aligned}
\ee
(ii) The case $j=i-1$ of (i) is given by
\be\label{eq:ejei}
e(i)J_{i-1}e(i)=(\delta-s+2i)J_ie(i).
\ee
This is consistent with \eqref{eq:eje} if $e_{k,l}$ is interpreted as $0$ for $k<0$.
\newline
(iii) With the above notation, we have, for $k=0,1,\dots,i-1$,

\be\label{eq:ejebig}
\begin{aligned}
e(1)J_0e(1)e(2)\dots e(k)=&A_k J_1e(1)e(2)\dots e(k)\\
                        &+B_kJ_1 e(1)e(2)\dots e(k+1)J_{k},
\end{aligned}
\ee
where
\be\label{eq:akbk}
\begin{aligned}
A_k=&k(\delta-s+2)+k(k-1), \text{ and}\\
B_k=&\frac{1}{(i-k-1)!(s-i-k-1)!}\\
\end{aligned}
\ee
(iv) The statement (iii) remains true for $k=i$, given the conventions
for interpreting $a(p,q)(=1)$ when $p\geq q$ and $e_{p,q}(=0)$ when $p<0$.
That is,
\be\label{eq:ejebigi}
e(1)J_0e(1)e(2)\dots e(i)=A_iJ_1e(1)e(2)\dots e(i),
\ee
where $A_i=i(\delta-s+2)+i(i-1)$ is as given by the formula \eqref{eq:akbk}.
\end{corollary}
\begin{proof}
For $j<i-1$, the statement (i) is simply a translation of Lemma \ref{lem:comp1} into the
present context. When $j=i-1$ or $i$, straightforward calculation shows that the
formula \eqref{eq:eje} remains true, given the specified conventions. This proves
(i) and (ii).

The assertion (iii) is proved by induction on $k$. First observe that
the case $k=0$ asserts that
$e(1)J_0=B_0J_1e(1)J_0$, where $B_0=\frac{1}{(i-1)!(s-i-1)!}$, which is easily
checked. Now suppose that for some fixed $k$ with $0\leq k\leq i-1$, we have
$$
e(1)J_0e(1)\dots e(k)=A_k J_1e(1)\dots e(k)+B_kJ_1 e(1)\dots e(k+1)J_{k}.
$$
Multiplying on the right by $e(k+1)$ and applying \eqref{eq:eje} to the second
term, which terminates with $e(k+1)J_ke(k+1)$, a short calculation
shows that we obtain
$$
e(1)J_0e(1)\dots e(k+1)=A_{k+1} J_1e(1)\dots e(k+1)+B_{k+1}J_1 e(1)\dots e(k+2)J_{k+1},
$$
where
\be\label{eq:recursion}
\begin{aligned}
A_{k+1}&=A_k+B_k(\delta-s+2k+2)(i-k-1)!(s-i-k-1)!\text{ and}\\
B_{k+1}&=B_k(i-k-1)(s-i-k-1).\\
\end{aligned}
\ee

The recursion \eqref{eq:recursion} has the unique solution given in \eqref{eq:akbk}.

Finally, (iv) follows from the case $k=i-1$ of (iii), noting that $J_i=a(2i+1,s)$,
and that the case $j=i-1$ of (i) yields that $e(i)J_{i-1}e(i)=(\delta-s+2i)e(i)J_i$.
The proof of this last formula involves separate consideration of the cases $s=2i$ and
$s>2i$.
\end{proof}

\subsection{A computation in the symmetric group algebra}

Fix an integer $s$ and for $i$ such that $0\leq i\leq s-i$, define
$F_i(s):=a(1,i)a(i+1,s)$, with the usual conventions.

\begin{lemma}\label{lem:lr}
Write $K\Sym_r=\oplus_{\lambda}I_\lambda$ for the usual canonical decomposition
of the group algebra of $\Sym_r$ into simple two-sided ideals.
Then
\begin{enumerate}
\item The ideal $\langle F_i(s)\rangle$ of $K\Sym_r$ generated by $F_i(s)$ is the
sum of those $I_\lambda$ such that $\lambda$ has at least $s-i$ boxes in its first
column and $s$ boxes in its first and second column.
\item We have, for $1\leq i\leq \frac{s}{2}$,
$\langle F_{i-1}(s)\rangle\subseteq\langle F_i(s)\rangle$.
\item There exist elements $\alpha_i,\beta_i\in B_r(\delta)$ such that for
$i$ as in (ii), $F_{i-1}=\alpha_i F_i\beta_i$
\end{enumerate}
\end{lemma}
\begin{proof}
The first statement follows easily from the Littlewood-Richardson rule. In fact one only
requires the (dual of) the Pieri rule. The second and third statements follow
easily from (i).
\end{proof}

\subsection{The annihilation theorem}

We shall prove the following result.
\begin{theorem}\label{thm:ann}
Let $E_i\in B_r(n)$, $i=0,1,\dots,\half$ be the elements defined
in Definition \ref{def:br-elts}(iii), and assume that $r\geq n+1$. Then
for $j=1,2,\dots,n$, we have $e_jE_i=E_ie_j=0$.
\end{theorem}
\begin{proof}
It is clear from Lemma \ref{lem:capzero} that $e_jF_i=F_ie_j=0$ for
$j\in\{1,2,\dots,n+1\}$, $j\neq i$. Hence to prove the theorem
it suffices to prove that
\be\label{eq:eezero}
e_iE_i=E_ie_i=0.
\ee
Moreover, since we have $E_I^*=E_i$ and $e_i^*=e_i$ where $^*$ is the
cellular involution of $B_r(n)$ (reflection in a horizontal),
to prove \eqref{eq:eezero}, it suffices to prove that
$e_iE_i=0$, since this implies that $(e_iE_i)^*=E_i^*e_i^*=E_ie_i=0$.
Thus we are reduced to proving that $e_iE_i=0$.

Maintaining the notation of Definition \ref{def:br-elts}, define elements
$F_i(k)=a(1,i-k)a(i+k+1,n+1)$ for $k=0,1,\dots,i$, with the usual conventions
applying. Thus $F_i(0)=F_i$. The $F_i(k)$ are analogues of the elements $J_k$
of Corollary \ref{cor:basic}, and translating (iii) of that Corollary
into the notation of the elements in Definition \ref{def:br-elts}
we obtain, bearing in mind that here $\delta=n$ and $s=n+1$,
$$
e_iF_ie_i(k)=k^2F_i(1)e_i(k)+\frac{1}{(i-k-1)!(n-i-k)!}F_i(1)e_i(k+1)F_i(k),
$$
for $k=0,1,2,\dots,i$. Note that when $k=0$ the first term vanishes,
and when $k=i$, the second term vanishes, given our conventions.

It follows that for $k=0,1,2,\dots,i$, with the usual notaional conventions,
\be\label{eq:eitermk}
e_iF_ie_i(k)F_i=k^2F_i(1)e_i(k)F_i+(i-k)(n+1-i-k)F_i(1)e(k+1)F_i.
\ee

It follows from \eqref{eq:eitermk} that $e_iE_i$ is a linear combination of
$F_i(1)e_i(j)F_i$, for $j=1,2,\dots,i$. Moreover, also by
\eqref{eq:eitermk}, the coefficient of $F_i(1)e_i(k)F_i$ in $e_iE_i$
is, in the notation of Definition \ref{def:br-elts}(iii),
$$
(-1)^k\left(k^2c_i(k)-(i-k+1)(n+2-i-k)c_i(k-1)\right),
$$
and using the
explicit values of the $c_i(j)$, this is equal to
$$
\begin{aligned}
(-1)^k\{\frac{1}{(i-k)!(n+1-i-k)!(k-1)!^2}&\\
-(i-k+1)(n+2-i-k)&\frac{1}{(i-k+1)!(n+2-i-k)!(k-1)!^2}\},\\
\end{aligned}
$$
which is equal to zero.

This shows that $e_iE_i=0$, and hence completes the proof of the theorem.
\end{proof}

\begin{corollary}\label{cor:ann}
Let $D$ be any diagram in $B_{n+1}(n)\subseteq B_r(n)$ which has fewer
than $n+1$ through strings (i.e., which has a horizontal arc). Then
for $i=0,1,2,\dots,\half$, $DE_i=E_iD=0$.
\end{corollary}
\begin{proof}
Fix $i$ as above. If $\sigma\in\Sym_i\times\Sym_{n+1-i}\subset B_{n+1}(n)$,
then $\sigma E_i=\pm E_i$. Now it is clear that for any diagram $D$ as above,
there is an element $\sigma\in\Sym_i\times\Sym_{n+1-i}$ such that for some
$j\in\{1,2,\dots,n\}$, $D\sigma=D'e_j$ for some diagram $D'\in B_{n+1}(n)$.
Hence $DE_i=\pm D\sigma E_i=\pm D'e_jE_i$, which is zero by Theorem \ref{thm:ann}.
The proof that $E_iD=0$ is similar.
\end{proof}

\begin{corollary}\label{cor:idemp}
The elements $E_i$ are quasi-idempotent. Specifically, we have, for
$i=0,1,\dots,\half$,
$$
E_i^2=i!(n+1-i)!E_i.
$$
\end{corollary}
\begin{proof}
Recall that $E_i=\sum_{j=0}^i(-1)^jc_i(j)F_ie_i(j)F_i$.
Now it follows from Corollary \ref{cor:ann} that for $j>0$,
$E_iF_ie_i(j)F_i=0$, since the second factor is a sum of diagrams
with at least one horizontal arc. Hence
$E_i^2=E_i\sum_{j=0}^i(-1)^jc_i(j)F_ie_i(j)F_i=E_iF_i=i!(n+1-i)!E_i.$
\end{proof}

\section{Generators of the kernel}

In this section we shall prove

\begin{theorem}\label{thm:kergens}
The ideal $\ker(\nu)$ is generated by $E_0,E_1,\dots,E_\half$.
\end{theorem}

%\section{The case of positive deficiency}
%In this section we shall show that the elements $E_i$ generate $\ker( \nu)$, and
%hence that $E=E_\half$ generates $\ker(\nu)$. This will be done by showing that
%each of the elements $b(S,S',\beta)$ lies in the ideal of $B_r(n)$ generated by
%a certain explicit element $E_{i,j}$, and that each element $E_{i,j}$ lies in the ideal
%genrated by $E_k$ for some $k$.
%\subsection*{Generators for $\ker(\nu)$}

\subsection{The cellular anti-involution on $B_r(\delta)$}

We shall make use of the cellular anti-involution \cite{GL96} on $B_r(\delta)$. This
is the unique algebra anti-involution $^*:B_r(\delta)\to B_r(\delta)$
satisfying $s_i^*=s_i$ and $e_i^*=e_i$ for each $i$. Then for $\sigma\in\Sym_r$,
we have $\sigma^*=\sigma\inv$. Geometrically, $^*$ may be thought of as reflecting
diagrams in a horizontal line.

The proof of Theorem \ref{thm:kergens} will proceed by
showing that each of the elements $b(S,S',\beta)$ lies in the
ideal of $B_r(n)$ generated by
a certain explicit element $E_{ij}$ or $E_{ij}^*$,
and that each of the elements $E_{ij}$ and $E_{ij}^*$ lies in the ideal
generated by $E_k$ for some $k$.

\subsection{Generators and deficiency}

Let $S,S'$ be disjoint subsets of $\{1,\dots,2r\}$ such that $|S|=|S'|=n+1$.
If $|S\cap \{1,\dots,r\}|=i$
and $|S'\cap \{1,\dots,r\}|=j$, then after pre and post multiplying
$b(S,S',\beta)$ by elements of $\Sym_r\subset B_r(n)$
and possibly interchanging $S$ with $S'$ and $\{1,\dots,r\}$ with $\{r+1,\dots,2r\}$,
we may assume that: $i\leq j$ and $i+j\geq n+1$;
the case $i+j=n+1$ leads to the elements $E_i$.
We write $d_{ij}=i+j-(n+1)$ and refer to this as the {\it deficiency} of the pair $S,S'$.
Fix $i,j$ as above.

For $k=0,1,\dots,n+1-j$, let $D_{ij}(k)$ be the diagram depicted in Figure \ref{fig10},
in which the points of $S$ are denoted
by $\circ$, those of $S'$ by $*$ and those of $\{1,\dots,2r\}\setminus(S\amalg S')$
by $\bullet$.

\begin{figure}
\begin{center}
\begin{tikzpicture}[scale=0.78]

%% draw filled dots at required points at bottom
\foreach \x in {1,3,4,6}
\draw node at (\x,0){*};
\foreach \x in {7,9,10,12}
\draw(\x,0) circle (0.1cm);
\foreach \x in {13,15,16,18}
\filldraw(\x,0) circle (0.1cm);
%% draw filled dots at required points on top
\foreach \x in {1,3,4,7,9}
\draw(\x,3) circle (0.1cm);
\foreach \x in {10,12,15,16,18}
\draw node at (\x,3){*};
%\foreach \x in {11,14}
%\filldraw(\x,2) circle (0.1cm);

%% draw straight lines where required
\foreach \x in {1,3}
\draw (\x,3)--(\x,0);
\draw (16,3)--(10,0);
\draw (18,3)--(12,0);

%% label the points in top row
\draw node[above] at (1,3){1};\draw node[above] at (3,3){\tiny{$i-d_{ij}-k$}};
\draw node[above] at (7,3){\tiny{$n+2-j$}};
\draw node[above] at (9,3){\tiny{$i$}};\draw node[above] at (10,3){\tiny{$i+1$}};
\draw node[above] at (12,3){\tiny{$i+d_{ij}$}};
\draw node[above] at (15,3){\tiny{$i+d_{ij}+k$}};
\draw node[above] at (18,3){\tiny{$i+j$}};

%% label the points in bottom row
\draw node[below] at (1,0){\tiny{$r+1$}};
\draw node[below] at (3,0){\tiny{$\overset{r+i}{-d_{ij}-k}$}};
\draw node[below] at (7,0){\tiny{$\overset{r+n+2}{-j}$}};
\draw node[below] at (9,0){\tiny{$\overset{r+n}{+1-j+k}$}};
\draw node[below] at (12,0){\tiny{$\overset{r+n+1}{-d_{ij}}$}};
%\draw node[below] at (13,0){\tiny{$\overset{r+n+2}{-d_{ij}}$}};
\draw node[below] at (15,0){\tiny{$r+n+1$}};\draw node[below] at (18,0){\tiny{$r+i+j$}};

%% draw curved lines where required, top
\draw(4,3) .. controls (7.5,1.5) and (11.5,1.5) .. (15,3);
\draw(7,3) .. controls (8.5,2.1) and (10.5,2.1) .. (12,3);
\draw(9,3) .. controls (9.2,2.5) and (9.8,2.5) .. (10,3);

%% draw curved lines where required, bottom
\draw(4,0) .. controls (5,1.4) and (8,1.4) .. (9,0);
\draw(6,0) .. controls (6.2,0.5) and (6.8,0.5) .. (7,0);
%\draw(6,0) .. controls (7,1.5) and (10,1.5) .. (11,0);
\draw(13,0) .. controls (14,1.4) and (17,1.4) .. (18,0);
\draw(15,0) .. controls (15.2,0.5) and (15.8,0.5) .. (16,0);

%% draw ... between straight lines etc
\draw(2,1.5) node {\large{$\cdots$}};\draw(5,0) node {{$\cdots$}};
\draw(2,0) node {{$\cdots$}};\draw(2,3) node {{$\cdots$}};
\draw(8,0) node {{$\cdots$}};\draw(17,3) node {{$\cdots$}};\draw(11,0) node {{$\cdots$}};
\draw(5.5,3) node {{$\cdots$}};\draw(8,3) node {{$\cdots$}};\draw(11,3) node {{$\cdots$}};
\draw(14,1.5) node {\large{$\cdots$}};\draw(14,0) node {{$\cdots$}};\draw(17,0) node {{$\cdots$}};
\draw(13.5,3) node {{$\cdots$}};
%\draw(12.5,1.5) node {\large{$\cdots$}};
\end{tikzpicture}
%\centerline{Figure 6}
\end{center}
\caption{ }
\label{fig10}
\end{figure}

The diagram $D_{ij}(k)$ is regarded as an element of $B_r(n)$ through
the natural inclusion $B_l(n)\hookrightarrow B_r(n)$ for any $l\leq r$.
Note that in the deficiency zero case, where $i+j=n+1$, the diagram $D_{ij}(k)$
coincides with the diagram $e_i(k)$ of definition \ref{def:br-elts}(ii).

\begin{definition}\label{def:eijk}
For $i,j$ such that $0\leq i\leq j\leq n+1$ and $i+j\geq n+1$, define
$E_{ij}\in B_r(n)$ by
\be\label{eq:eij}
E_{ij}=\sum_{k=0}^{n+1-j}(-1)^jc_{ij}(k)a(1,i)a(i+1,i+j)D_{ij}(k)
a(1,n+1-j)a(n+2-j,n+1-d_{ij}),
\ee
where $c_{ij}(k)=\left( (n+1-j-k)!(n+1-i-k)!k!(d_{ij}+k)!\right)\inv$.
\end{definition}

\begin{proposition}\label{prop:genseij}
(i) The elements $E_{ij}$ and $E_{ij}^*$ are in $\ker(\nu)$.

(ii) The kernel $\ker(\nu)$  is generated as an ideal of $B_r(n)$ by the elements $E_{ij}$
and $E_{ij}^*$, where $0\leq i\leq j\leq n+1$, $i+j\geq n+1$.
Here $^*$ denotes the cellular involution of $B_r(n)$, discussed above.
\end{proposition}
\begin{proof} The kernel $\ker(\nu)$ is spanned by the elements $b(S,S',\beta)$,
each of which is an alternating sum over $\Sym_{n+1}$. To see that $E_{ij}\in\ker(\nu)$,
we shall show that $E_{ij}$ is precisely one of the elements $b(S,S',\beta)$,
where $S=S_{ij}:=\{1,\dots,i,r+i+1-d_{ij},\dots,r+n+1-d_{ij}\}$,
$S'=S'_{ij}:=\{i+1,\dots,i+j,r+1,\dots,r+i-d_{ij}\}$,
and $\beta$ is the pairing of $\{1,\dots,2r\}\setminus(S\amalg S')$ depicted in
the diagram $D_{ij}(k)$ for any $k$.

Observe that from the formula \eqref{eq:eij}, $E_{ij}$ is alternating with respect to
both $\Sym(S)$ and $\Sym(S')$, for if $t$ is any transposition in $\Sym(S)$,
then $t\cdot E_{ij}=-E_{ij}$, and similarly for $S'$. In fact, the constants
$c_{ij}(k)$ are chosen so that $E_{ij}$ is precisely the alternating sum of
$(n+1)!$ diagrams, obtained from $D_{ij}(0)$ by permuting the elements of $S$.
This shows that $E_{ij}\in\ker(\nu)$. Since $\ker(\nu)$ is evidently invariant
under $^*$, this proves (i).

It is straightforward to see that using the action of $\Sym_r$ on the right and left, any
pair $S,S'$ of subsets as above may be transformed into a pair
$S_{ij},S_{ij}'$ or $S_{ij}',S_{ij}$, where these sets are as above, with $i\leq j$
and $i+j\geq n+1$.

It follows, since any summand of an element $b(S,S',\beta)$, where the
pair $S,S'$ has deficiency $d$ has at least $d$ horizontal edges, that any
element $b(S,S',\beta)$ may be transformed by $\Sym_r\times\Sym_r$
into an element of $B_r(n)$ each of whose diagram summands
satisfies the condition that its leftmost $i+j$ part coincides with that
of $E_{ij}$ or $E_{ij}^*$ for some $i,j$, and whose
rightmost $r-(i+j)$ part is constant for each such summand.
Hence $\pi b(S,S',\beta)\pi'=E_{ij}D_\alpha$
or $E_{ij}^*D_\alpha$ for some $\pi,\pi'\in\Sym_r$ and $D_\alpha\in B_r(n)$. This proves (ii).
\end{proof}

\begin{lemma}\label{lem:pi}
We have, for each $i,j$ and $k$ as above,
$$
D_{ij}(k)=e_i(k+d_{ij})\pi_{ij}=e_{i,i+1}e_{i-1,i+2}\dots e_{i-k-d_{ij}+1,i+k+d_{ij}}\pi_{ij},
$$
and
$$
E_{ij}=\sum_{k=0}^{n+1-j}(-1)^jc_{ij}(k)a(1,i)a(i+1,i+j)e_i(k+d_{ij})a(1,n+1-j)a(2i+j-n,i+j)\pi_{ij},
$$
where $\pi_{ij}\in\Sym_r\subset B_r(n)$ is the permutation defined by
$$
\pi_{ij}(l)=
\begin{cases}
l\text{ if $1\leq l\leq n+1-j$ or $l>i+j$}\\
l+n+1-i\text{ if $i-d_{ij}+1\leq l\leq i+d_{ij}$}\\
l-2d_{ij} \text{ if $i+d_{ij}+1\leq l\leq i+j$}.\\
\end{cases}
$$
The permutation $\pi_{ij}$ is independent of $k$.
\end{lemma}
\begin{proof}
The first statement may be directly verified, and the second follows easily, by
computing $\pi_{ij}a(1,n+1-j)a(2i+j-n,i+j)\pi_{ij}\inv$.
\end{proof}

The next result is required for the proof of Theorem \ref{thm:kergens}.

\begin{theorem}\label{thm:eloopzero}
We have, in the above notation, $e_\ell E_{ij}=0$ for
all $\ell$ such that $1\leq\ell\leq i+j-1$.
\end{theorem}
\begin{proof}
It is clear by Lemma \ref{lem:capzero} that the Theorem holds for $\ell\neq i$.
It therefore suffices to prove that
\be\label{eq:eieijzero}
e_iE_{ij}=0.
\ee
To apply the computations of \S \ref{ss:basic},
it is convenient to rewrite the $E_{ij}$ as follows.
For $d$ in the range $0\leq d\leq i$,
write $F_{ij}(d)=a(1,i-d)a(i+1+d,i+j)$.
Noting that $n+1-j=i+d_{ij}$, etc., we may rewrite the expression for $E_{ij}$
in Lemma \ref{lem:pi} as
\be\label{eq:neweij}
E_{ij}=\sum_{k=0}^{n+1-j}(-1)^jc_{ij}(k)F_{ij}(0)e_i(k+d_{ij})F_{ij}(d_{ij})\pi_{ij}.
\ee

Note that the elements $F_{ij}(d)$ are special cases of the elements $J_d$
of Corollary \ref{cor:basic}, which may now be applied directly, replacing
$\delta, s$ and $k$ respectively by $n, i+j$ and $k+d_{ij}=i+j+k-(n+1)$.

We obtain

$$
\begin{aligned}
e_iF_{ij}(0)&e_i(k+d_{ij})F_{ij}(d_{ij})=A_{k+d_{ij}}F_{ij}(1)e_i(k+d_{ij})F_{ij}(d_{ij})\\
+&B_{k+d_{ij}}F_{ij}(1)e_i(k+1+d_{ij})F_{ij}(k+d_{ij})F_{ij}(d_{ij})\\
=&A_{k+d_{ij}}F_{ij}(1)e_i(k+d_{ij})F_{ij}(d_{ij})\\
+&(i-k-d_{ij})!(j-k-d_{ij})!B_{k+d_{ij}}F_{ij}(1)e_i(k+1+d_{ij})F_{ij}(d_{ij}).\\
\end{aligned}
$$

It follows that in the expression for $e_iE_{ij}$ as a sum of the elements
$G_k:=F_{ij}(1)e_i(k)F_{ij}(d_{ij})\pi_{ij}$, the coefficient of $G_{k+1}$
is

$$
(-1)^{k+1}\left(A_{k+1+d_{ij}}c_{ij}(k+1)-(i-k-d_{ij})!(j-k-d_{ij})!
B_{k+d_{ij}}c_{ij}(k)\right).
$$
To evaluate this we substitute the actual values of $A_\ell$ and $B_\ell$.
We have
$$
\begin{aligned}
A_{k+d_{ij}}=&(k+d_{ij})(n-(i+j)+2)+(k+d_{ij})(k+d_{ij}-1)\\
=&k(k+d_{ij}),\\
\end{aligned}
$$
while
$$
B_{k+d_{ij}}=\left[(i-1-(k+d_{ij}))!(j-1-(k+d_{ij}))!\right]\inv.
$$

Moreover
$$
c_{ij}(k)=\left[(d_{ij}+k)!k!(n+1-j-k)!n+1-i-k)!\right]\inv.
$$

Substituting these values into the expression above, we obtain
$$
\begin{aligned}
A_{k+1+d_{ij}}c_{ij}(k+1)=&\frac{1}{(d_{ij}+k)!k!(n-j-k)!(n-i-k)!}\\
=&(i-k-d_{ij})!(j-k-d_{ij})!B_{k+d_{ij}}c_{ij}(k).\\
\end{aligned}
$$
It follows that the coefficient of
$G_k$ in $e_iE_{ij}$ is zero, for $k=0,1,\dots,i-d_{ij}$.
Hence $e_iE_{ij}=0$, and the Theorem is proved.
\end{proof}

\begin{corollary}\label{cor:deijzero}
If $D$ is any diagram in $B_{i+j}(n)$ with at least one horizontal edge,
then $DE_{ij}=0$.
\end{corollary}
\begin{proof}
For any such diagram $D$, there is a permutation $\sigma\in\Sym_i\times\Sym_j$
such that $D\sigma=D'e_\ell$ for some diagram $D'\in B_{i+j}(n)$ and $\ell$
satisfying $1\leq \ell\leq i+j-1$. The result now follows from Theorem
\ref{thm:eloopzero}.
\end{proof}

\begin{corollary}\label{cor:eijinei}
We have, for each pair $i,j$ with $i\leq j$ and $i+j\geq n+1$
$E_{ij}\in\langle E_k\rangle$ for some $k$ with $0\leq k\leq\half$.
We also have $E_{ij}^*\in\langle E_k\rangle$ for the same $k$.
\end{corollary}
\begin{proof}
It follows from Corollary \ref{cor:deijzero} that for any elements
$x,y\in B_{i+j}(n)$ and any $k$ such that $0\leq k\leq\half$,
we have $xE_ky E_{ij}=xF_ky E_{ij}$, since the other summands of
the product vanish. Write $F_{ij}=F_{ij}(0)$ in the notation of the
proof of Theorem \ref{thm:eloopzero}. By the above observation,
since $F_{ij}E_{ij}$ is a non-zero multiple of $E_{ij}$,
in order to show that $E_{ij}\in\langle E_k\rangle$, it will suffice
to show that there are elements $x,y\in B_{i+j}(n)$ such that
$xF_ky=F_{ij}$ for some $k$. We shall in fact show that there are
elements $x,y$ of $K\Sym_{i+j}\subset B_{i+j}(n)\subseteq B_r(n)$ which have the desired
property.

Now we have seen that in $K\Sym_{i+j}$, the ideal $\langle F_{ij}\rangle
=\oplus_\lambda I_\lambda$, where $\lambda$ runs over partitions of
$i+j$ whose first column has at least $j$ elements, and whose first
two columns have at least $i+j$ elements. But for $k=0,1,\dots,\half$,
$\langle F_k\rangle=\oplus_\mu I_mu$, where $\mu$ runs over partitions
whose first column contains at least $n+1-k$ elements, and whose first
two columns contain at least $n+1$ elements. Since $i+j\geq n+1$,
it follows that for $k\geq n+1-j$ (note that
$i+j\geq n+1,i\leq j\implies n+1-j\leq\half$), $E_{ij}\in\langle F_k\rangle$,
whence there are elements $x,y\in K\Sym_{i+j}$ such that $E_{ij}=xF_ky$, whence
$E_{ij}\in\langle E_k\rangle$.

To show that $E_{ij}^*\in\langle E_k\rangle$, observe that by taking the $^*$ of
Corollary \ref{cor:deijzero}, we have $E_{ij}^*D=0$ of any diagram in $B_{i+j}(n)$
with at least one horizontal edge. Hence as above, we see that for any elements
$x,y\in B_{i+j}(n)$, $E_{ij}^*xE_ky=E_{ij}^*xF_ky$, and the argument
proceeds as above.
This completes the proof of the Corollary.
\end{proof}

We may now complete the
\begin{proof}[Proof of Theorem  \ref{thm:kergens}]
It follows from Proposition \ref{prop:genseij} that $\ker(\nu)$ is generated by
the $E_{ij}$ and $E_{ij}^*$. But by Corollary \ref{cor:eijinei} each
of the elements $E_{ij}$ and $E_{ij}^*$
is in the ideal generated by $E_0,E_1,\dots,E_\half$.
Theorem \ref{thm:kergens} follows.
\end{proof}

\section{Proof of the main theorem}

In this section we complete the proof of Theorem \ref{thm:main}. The arguments
are similar to the ones employed in the last section.

\begin{proof}[Proof of Theorem \ref{thm:main}]
The first assertion of the Theorem is a special case of Corollary \ref{cor:idemp}.
To prove that $E=E_\half$ generates $\ker(\nu)$ we proceed as follows. By Proposition
\ref{prop:einker}, the $E_i$ are in $\ker(\nu)$, and by Theorem \ref{thm:kergens}
$ker(\nu)$ is generated by the $E_i$.
The result will therefore follow if we show that
\be\label{eq:contains}
\text{For $i=1,\dots,\half$, $E_{i-1}$ is in the ideal generated by $E_{i}$.}
\ee
For if \eqref{eq:contains} holds, then writing $\langle y\rangle$ for the ideal
of $B_r(n)$ generated by any element $y\in B_r(n)$ we would have
$\langle E\rangle\supseteq\langle E_{\half-1}\rangle\supseteq\dots\supseteq
\langle E_1\rangle\supseteq\langle E_0\rangle$.

To prove \eqref{eq:contains}, let $\alpha_i,\beta_i$ be elements of
$B_r(n)$ as in Lemma \ref{lem:lr}. Then for any
$i$ such that $i\leq \frac{n+1}{2}$, $F_{i-1}=\alpha_i F_i\beta_i$.
Consider the element $x:=E_{i-1}\alpha_iE_i\beta_i\in\langle E_i\rangle$.
Now $\alpha_iE_i\beta_i=\sum_{j=0}^i\alpha_iE_i(j)\beta_i$, where
$E_i(j)=(-1)^jc_i(j)F_ie_i(j)F_i$ is a sum of diagrams in $B_{n+1}(n)$
with at least $j$ horizontal arcs. Hence by Corollary \ref{cor:ann},
$E_{i-1}\alpha_iE_i(j)\beta_i=0$ if $j>0$.

It follows that $x=E_{i-1}\alpha_iF_i\beta_i$, since $E_i(0)=F_i$.
Hence $x=E_{i-1}F_{i-1}=(i-1)!(n-i)!E_{i-1}\in\langle E_i\rangle$.
This proves \eqref{eq:contains}, and completes the proof of Theorem
\ref{thm:main}.
\end{proof}

\section{Cellular structure}

It is well known that $B_r(n)$ has a cellular structure \cite[\S 4]{GL96} in which
the cells are indexed by the set $\Lambda$ of partitions
$\lambda=(\lambda_1\geq\dots\geq\lambda_p)$, with $|\lambda|=\sum_{i=1}^p\lambda_i\in\CT$,
where $\CT=\{t\in\Z\mid 0\leq t\leq r;\;\;t\equiv r\;(\mod 2)\}$. The partial order
on $\Lambda$ is given by the rule that $\lambda<\mu$ if $|\lambda|<|\mu|$, or
$|\lambda|=|\mu|$ and $\lambda<\mu$ in the dominance order.

We therefore have cell modules $W(\lambda)$ ($\lambda\in\Lambda$)
for $B_r(n)$. These are endowed with
a canonical invariant form, whose radical
$\Rad(\lambda)$ has irreducible quotient which
we write here as $I_\lambda$. It is part of the general theory of cellular algebras
that the non-zero $I_\lambda$ form a complete
set of representatives of the isomorphism classes
of $B_r(n)$. Now $B_r(\delta)$ is quasi-hereditary \cite{X}
whenever $\delta\neq 0$, from which it follows that
the $I_\lambda$ are each non-zero, and therefore that
the irreducible $B_r(n)$-modules are indexed by $\Lambda$.

Now we have assumed that the characteristic of $K$ is zero. In consequence,
$V^{\otimes r}$ is semisimple, both as $\Or(V)$-module, and as $B_r(n)$-module.

\begin{lemma}\label{lem:doubcent}
There is a subset $\Lambda^0$ of $\Lambda$ such that as $\Or(V)\times B_r(n)$-module
$$
V^{\otimes r}\simeq\oplus_{\lambda\in\Lambda^0}L_\lambda\otimes I_\lambda,
$$
where $\Lambda^0$ is the subset of $\Lambda$ consisting of partitions
whose first and second
columns have fewer than $n+1$ elements in total. Here $L_\lambda$ and $I_\lambda$
are respectively the simple $\Or(V)$-module and the simple $B_r(V)$-module
corresponding to $\lambda$.
\end{lemma}
\begin{proof}
The statement that there is a decomposition of the type shown, for some
subset of $\Lambda$, follows from generalities about double centraliser theory.
The identification of $\Lambda^0$ in our case follows easily from
our Theorem \ref{thm:main}, but is in any case well known.
\end{proof}

One consequence of this lemma is the the multiplicity of $L_\lambda$ in
$V^{\otimes r}$ is the dimension of $I_\lambda$. However in consequence of
the non-semisimple nature of $B_r(n)$ for $r\geq n+2$, these dimensions are
not given by purely combinatorial (cellular) data. Nonetheless our main theorem
is relevant to this decomposition through the following result.

\begin{proposition}\label{prop:cells}
Let $E=E_\half$ be the element of $B_r(n)$ defined in Definition \ref{def:br-elts}(iii).
Then for $\lambda\in\Lambda^0$, the submodule $\Rad(\lambda)$ of $W(\lambda)$
is given by $\Rad(\lambda)=B_r(n)EW(\lambda)$. That is, $\Rad(\lambda)$ is
generated by $EW(\lambda)$.
\end{proposition}
\begin{proof}
First, note that by Lemma \ref{lem:doubcent},
$I_\lambda$ is a summand of $V^{\otimes r}$ if and only if
$\lambda\in\Lambda^0$. Hence $\Lambda^0$ consists of those $\lambda$
such that $I_\lambda$ is annihilated by $\ker(\nu)$, and hence by
$E$ since by Theorem \ref{thm:main} $E$ generates $\ker(\nu)$.

It follows that $EW(\lambda)\subseteq \Rad(\lambda)$. But the ideal
$\langle E\rangle$ contains the radical of the algebra $B_r(n)$.
Hence by the local criterion proved in \cite[Theorem 5.4(3)]{LZ2}
for a self dual ideal to contain the radical, it follows that for all
$\lambda\in \Lambda$, $\langle E\rangle W(\lambda)\supseteq \Rad(\lambda)$.
It follows that for $\lambda\in\Lambda^0$,
$\langle E\rangle W(\lambda)=\Rad(\lambda)$, and the Proposition follows.
\end{proof}

We conclude this section with the remark that by the above Proposition,
we have, for $\lambda\in\Lambda^0$,
$$
I_\lambda\simeq \frac{W(\lambda)}{\Rad(\lambda)}\simeq\frac{W(\lambda)}{
\langle E\rangle W(\lambda)},
$$
and this makes it possible in principle to compute the dimension
of $I_\lambda$ by identifying the subspace $EW(\lambda)$ of $W(\lambda)$.

\section{Change of base field, and remarks about the quantum case}

In this section we discuss the situation when the field $K$ has positive characteristic,
as well as the quantum analogue of our result, which applies to the Birman-Wenzl-Murakami
(BMW) algebra.

\subsection{The case of positive characteristic}

When $K$ is an arbitrary field, our basic setup remains the same. We still have
the map $\nu:B_r(n)\to \End_{\Or(V)}(V^{\otimes r})$, and it is still surjective.
Proposition \ref{prop:commute} also remains true.

It is important to note that although
the elements $E_i$ and $E_{ij}$ have denominators in their definitions, they are actually
linear combinations of diagrams with coefficients $\pm 1$. Therefore they are elements of the
Brauer algebra over $\Z$, and may be thought of independently of the ground field.
Many of the results above remain true for arbitrary $K$.
For any commutative ring $R$, with $\delta\in R$, write $B_r^R(\delta)$ for the
Brauer algebra over $R$, which may be defined by its presentation as given
in Lemma \ref{lem:brprops}(iii); $B_r^R(\delta)$ is free over $R$, with basis the set of
Brauer diagrams. As usual, $B_r(n)=B_r^K(n)$.

\begin{lemma}\label{lem:spangenk}
Let $K$ be a field of characteristic other than two. The elements $b(S,S',\beta)\in B_r^K(n)$
of Lemma \ref{lem:ker1} span $\ker(\nu)$.
\end{lemma}

This is because the version of the second fundamental theorem in \cite[Prop. 21]{Ri}
is valid in this generality, and our statement follows from the commutativity
of the diagram \eqref{diag:triangle}.

Next we have

\begin{proposition}\label{thm:eijgenk} Let $K$ be a field of characteristic other than two.
Let $E_{ij}\in B_r(n)$ be the elements defined in Definition \ref{def:eijk}.
Then $\ker(\nu)$ is generated as ideal of $B_r(n)$ by the $E_{ij},E_{ij}^*$.
\end{proposition}
\begin{proof}
Note that although the definition of $E_{ij}$ as given involves denominators,
since $E_{ij}$ is actually one of the elements $b(S,S',\beta)$, it is a $\Z$-linear
combination of diagrams, and hence may be interpreted as an element of $B_r^\Z(n)$,
and hence of $B_r^R(n)$ for any ring $R$.

The proof of the current proposition involves merely the observation
that the proof of Proposition \ref{prop:genseij}(ii) remains valid in this more general
setting.
\end{proof}

\begin{theorem}
Let $K$ be any ring, and let $E_{ij}\in B_r^K(n)$ be as in the previous proposition.
Then for $\ell$ with $1\leq \ell\leq i+j-1$, we have $e_\ell E_{ij}=0$.
\end{theorem}
\begin{proof}
Although the proof of Theorem \ref{thm:eloopzero} involves denominators, it is clear that
it may be restated (and proved in the same way) as a result in $B_r^\Z(n)$.
Applying the specialisation functor $K\ot_\Z$, we obtain the present statement.
\end{proof}

\begin{theorem}\label{thm:poschar} Let $K$ be a field of characteristic
$p>2(n+1)$. Then $\ker(\nu)=\langle E\rangle$, where $E$ is the element
$E_\half$ of Theorem \ref{thm:main}.
\end{theorem}

\begin{proof}
The proofs of Corollary \ref{cor:eijinei} and of \eqref{eq:contains}
involve computations in the group algebra $K\Sym_{i+j}$. But if the
characteristic of $K$ is greater than $i+j$, this algebra is semisimple,
and hence the arguments in those proofs apply without change.
The result follows.
\end{proof}

\begin{remark}
It is likely that the conclusion of the above theorem is valid for any characteristic
other than two.
\end{remark}

\subsection{The quantum case}

Let $U_q:=U_q(o_n)$ be the smash product of the quantised enveloping algebra corresponding to
the complex Lie algebra $so_n(\C)$ with the group algebra of $\Z_2$ (see \cite[\S 8]{LZ1}).
Let $\CC_q$ be the category of finite
dimensional type $(1,1,\dots,1)$ representations of  $U_q$.
Using Lusztig's integral form \cite{Lu}
of $U_q$ and lattices in the simple $U_q$-modules, we have a specialisation
functor $S:M_q\mapsto M$ taking modules $M_q$ in $\CC_q$ to their `classical limit'.

Let $V_q$ be the `natural' representation
of the quantum group $U_q(o_n)$, that is, the
representation which corresponds to the natural representation $V$
of $O_n$ under the specialisation above.
It is well known (cf., e.g. \cite{LZ1}) that there is a surjective
homomorphism $\psi:\C\CB_r\lr\End_{U_q(o_n)}(V_q^{\ot r})$,
where $\CB_r$ is the $r$-string braid group, acting through the
generalised $R$-matrices.

Moreover this action factors through $BMW_r(q)=BMW_r(q^{2(1-n)}, q^2-q^{-2})$,
the Birman-Murukami-Wenzl
algebra over $\C(q)$ with the indicated parameters.
The specialisation
at $q=1$ (see \cite[Lemma 4.2]{LZ2}) of $BMW_r(q)$ is $B_r(n)$.
It follows from the results of \cite{LZ1,LZ2,LZ3}, that we have a
commutative diagram of specialisations as depicted below,
which compare the classical case, treated above, with the quantum case.

\be\label{eq:commdiag}
\begin{CD}
0 @>>> \ker(\psi) @>>>BMW_r(q) @>\psi>>\End_{U_q}(V_q^{\ot r}) @>>>0\\
@. @VVSV @VVSV @VVSV\\
0 @>>> \ker(\nu) @>>>B_r(n) @>\nu>>\End_{U_q}(V_q^{\ot r}) @>>>0\\
\end{CD}
\ee

This diagram naturally leads to the

\noindent{\bf Conjecture.} In the above notation, there is an element
$\Phi_q\in BMW_r(q)$ such that $\ker(\psi)=\langle \Phi_q\rangle$. The
specialisation at $q=1$ of $\Phi$ is $E$.

This was proved for the case $n=3$ in \cite[Theorem 2.6]{LZ3},
where an explicit formula was given for $\Phi_q$.

There is another way to generalise the result \cite[Theorem 2.7]{LZ3},
which is the case $n=3$ of our current work. It was shown in \cite{LZ1}
that the map $\C\CB_r\lr\End_{U_q(so_2)}(V_{d,q}^{\ot r})$ is surjective, where
$V_{d,q}$ is the $q$-analogue of the $d$-dimensional representation of
$U_q(so_2)$. It is natural to ask for presentations of finite dimensional
algebras through which this map factors (cf. \cite{Lu}). This is not likely to be
straightforward, because in this case, the generators satisfy a polynomial
equation of degree $d$.


\begin{thebibliography}{9999}
%
%
\bibitem{ABP} Atiyah, M.; Bott, R.; Patodi, V. K., ``On the heat
equation and the index theorem'', {\sl Invent. Math.  \bf 19}  (1973), 279--330.

\bibitem {BW} Birman, Joan and Wenzl, Hans,  ``Braids, link polynomials and a new algebra'', {\sl Trans. Amer. Math. Soc. \bf 313} (1989), 249--273.

\bibitem {Br} Richard Brauer, ``On algebras which are connected with the semisimple continuous groups'', {\sl Ann. of Math. (2) {\bf 38}}  (1937), 857--872.

\bibitem{DHW} William F. Doran IV, David B. Wales and Philip J. Hanlon, ``On the semisimplicity of the Brauer centralizer algebras'', {\sl J. Algebra  \bf 211}  (1999), 647--685.

\bibitem {GL96} J. Graham and G.I. Lehrer, ``Cellular algebras'', {\sl Inventiones Math. \bf 123} (1996), 1--34.

\bibitem{GL03} J.J.~Graham and G.~I.~Lehrer, ``Diagram algebras, Hecke algebras and decomposition numbers at roots of unity'' {\sl Ann. Sci. \'Ecole Norm. Sup. \bf 36}  (2003), 479--524.

\bibitem{GL04} J.J.~Graham and G.~I.~Lehrer, ``Cellular algebras and diagram algebras in representation theory'', {\sl Representation theory of algebraic groups and quantum groups, Adv. Stud. Pure Math. {\bf 40}}, Math. Soc. Japan, Tokyo, (2004), 141--173.

\bibitem{HX} Hu, Jun and Xiao, Zhankui, ``On tensor spaces for Birman-Murakami-Wenzl algebras'',{\sl J. Algebra \bf 324} (2010) 2893--2922.

\bibitem{Lo} Loday, Jean-Louis,  ``Cyclic homology'' Appendix E by María O. Ronco. Second edition. Chapter 13 by the author in collaboration with Teimuraz Pirashvili. {\sl Grundlehren der Mathematischen Wissenschaften [Fundamental Principles of Mathematical Sciences], \bf 301}, Springer-Verlag, Berlin, 1998. xx+513 pp.

\bibitem{LZ1} G.~I.~Lehrer and R.~B.~Zhang, ``Strongly multiplicity free modules for Lie algebras and quantum groups'', {\sl J. of Alg. \bf 306} (2006), 138--174.

\bibitem{LZ2} G.~I.~Lehrer and R.~B.~Zhang, ``A Temperley-Lieb analogue for the BMW algebra'', pp. 155--190,  in ``Representation Theory of Algebraic Groups and Quantum Groups", {\sl Progress in Math. \bf 284}, Birkh\"auser, Boston, MA, 2011.

\bibitem{LZ3} G.~I.~Lehrer and R.~B.~Zhang, ``On endomorphisms of quantum tensor space", {\sl Lett. Math. Phys. \bf 86} (2008), no. 2-3, 209--227.

\bibitem{LZZ} G.~I.~Lehrer, H.~Zhang and R.~B.~Zhang, ``A quantum analogue of the first fundamental theorem of classical invariant theory",  {\sl Comm. Math. Phys. \bf 301} (2011) no. 1, 131--174.

\bibitem{Lu} Lusztig, G. ``Introduction to quantum groups". Progress in Mathematics, 110. Birkhäuser Boston, Inc., Boston, MA, 1993.

\bibitem{P1} Procesi, Claudio, `150 years of invariant theory', {\sl The heritage of Emmy Noether} (Ramat-Gan, 1996), 5--21, {\sl Israel Math. Conf. Proc.,\bf 12}, Bar-Ilan Univ., Ramat Gan, 1999.

\bibitem{P2} Procesi, Claudio,  ``Lie groups. An approach through invariants and representations'', Universitext. Springer, New York, 2007. xxiv+596 pp.

\bibitem{P3} Procesi, C., ``The invariant theory of $n\times n$ matrices'', {\sl Advances in Math.  \bf 19}  (1976), no. 3, 306--381.

\bibitem{Ri} Richman, David R., ``The fundamental theorems of vector invariants'', {\sl Adv. in Math. \bf 73}  (1989),  no. 1, 43--78.

\bibitem{RS} Hebing Rui and Mei Si, ``A criterion on the semisimple Brauer algebras. II.'' {\sl J. Combin. Theory Ser. A {\bf 113}} (2006), 1199--1203.

\bibitem{W} H. Weyl, ``The Classical Groups. Their Invariants and Representations". Princeton University Press, Princeton, N.J., 1939.

\bibitem{X} Xi, Changchang.  ``On the quasi-heredity of Birman-Wenzl algebras". {\sl Adv. Math.  \bf 154}  (2000),  no. 2, 280--298.

\end{thebibliography}
\end{document}